\documentclass[12pt]{amsart}

\usepackage{hyperref}
\usepackage{amsmath,amsfonts,amsthm,amssymb,a4wide}

\def\ack{\subsubsection*{Acknowledgment.}}

\theoremstyle{plain}
\newtheorem{thm}{Theorem}[section]
\newtheorem{prop}[thm]{Proposition}
\newtheorem{lem}[thm]{Lemma}

\theoremstyle{definition}
\newtheorem{dfn}{Definition}[section]

\newcommand{\C}{{\mathbb{C}}}
\newcommand{\R}{{\mathbb{R}}}
\newcommand{\Z}{{\mathbb{Z}}}
\newcommand{\N}{{\mathbb{N}}}
\newcommand{\la}{\lambda}
\newcommand{\Si}{\Sigma}
\newcommand{\al}{\alpha}
\newcommand{\ga}{\gamma}

\begin{document}

\title{Spectrum of a Rudin--Shapiro-like sequence}

\author{Lax Chan}
\author{Uwe Grimm}

\address{School of Mathematics and Statistics, The Open University, 
Walton Hall, Milton Keynes MK7 6AA, United Kingdom
Email addresses: \{lax.chan,uwe.grimm\}@open.ac.uk}

\begin{abstract}
  We show that a recently proposed Rudin--Shapiro-like sequence, with
  balanced weights, has purely singular continuous diffraction
  spectrum, in contrast to the well-known Rudin--Shapiro sequence
  whose diffraction is absolutely continuous. This answers a question
  that had been raised about this new sequence.
\end{abstract}

\maketitle

\section{Introduction}

Substitution dynamical systems are widely used as toy models for
aperiodic phenomena in one dimension \cite{PF02}. Crystallographers
are interested in the diffraction spectrum of these systems because it
provides information about the structure of a material
\cite{BG12}. Dworkin~\cite{SD93} showed that the diffraction
spectrum is related to part of the dynamical spectrum, which is the
spectrum of a unitary operator acting on a Hilbert space, as induced
by the shift action. For recent developments regarding the relation
between diffraction and dynamical spectra, we refer to the review
\cite{BL16} and references therein.

The Rudin--Shapiro (RS) sequence \cite{Sha51,Rud59} (in its (balanced)
binary version with values in $\{\pm 1\}$) is a rare example of a
substitution-based system with purely absolutely continuous
diffraction spectrum (while its dynamical spectrum is mixed,
containing the dyadic integers as its pure point part); see
\cite{AS03} for background.  A `Rudin--Shapiro-like' (RSL) sequence
was recently introduced and analyzed in~\cite{PNR15}.  It is defined
as
\begin{equation}\label{eq:rsl}
  \textnormal{RSL}(n)\, =\, (-1)^{\textnormal{inv}^{}_2(n)},
\end{equation}
where $\textnormal{inv}^{}_{2}(n)$ counts the number of occurences of
$10$ (`inversions') as a scattered subsequence in the binary
representation of $n$. In \cite{PNR15}, it is shown that this
sequence exhibits some similar properties as the Rudin--Shapiro
sequence. In particular, this concerns the partial sums
$\varSigma(N):=\sum_{0\leq n\leq N}\textnormal{RSL}(N)$, which are
shown to have the form $\varSigma(N)=\sqrt{N}\,G(\log_{4}N)$, where
$G$ is a function that oscillates periodically between $\sqrt{3}/3$
and $\sqrt{2}$. At the end of \cite{PNR15}, the question is raised
whether this similarity between the two sequences extends to the
property that
\begin{equation}\label{eq:ineq}
  \sup_{\theta\in\R}\left|\sum_{n<N}RSL(n)\, e^{2\pi in\theta}\right| 
   \, \le\, C\, N^{\frac{1}{2}},
\end{equation}
which is satisfied by the Rudin--Shapiro sequence \cite{AL91}, and
which is linked to the purely absolutely continuous diffraction
measure of the balanced RS sequence.

In what follows, we are going to employ a recent algorithm by
Bartlett~\cite{AB14} to show that the Rudin--Shapiro-like sequence
has purely singular continuous diffraction spectrum, pointing to a big
structural difference to the Rudin--Shapiro sequence. In particular,
this will imply that Equation~\eqref{eq:ineq} does \emph{not} hold for the
Rudin--Shapiro-like sequence.

\section{A sketch of Bartlett's algorithm}

By generalizing and developing previous work of Queff\'{e}lec
\cite{MQ10}, Bartlett \cite{AB14} provides an algorithm that
characterizes the spectrum of an aperiodic, constant length
substitution $S$ on $\Z^{d}$. It describes the Fourier coefficients of
mutually singular measures of pure type, giving rise to the maximal
spectral type. Here, we can only give a brief sketch of Bartlett's
algorithm, concentrating on the case of dimension $d=1$.  

We assume that the substitution system is primitive. We first compute
the instruction matrices (or digit matrices) $R_j$, where $j\in[0,q)$
and $q$ is the length of the substitution (which will be $q=2$ in our
case). These matrices encode the letters that appear at the $j$-th
position of the image of the substitution system; we shall show this
for the explicit example of the Rudin--Shapiro-like sequence
below. The substitution matrix $M_S$ is given by the sum of the
instruction matrices.

Due to primitivity, the Perron--Frobenius theorem
\cite[Thm.~2.2]{BG13} ensures that the eigenvector to the leading
eigenvalue of $M_S$ can be chosen to have positive entries only. We
denote this vector, after normalizing it to be a probability vector,
by $u$. Note that $u=(u_{\ga})_{\ga\in\mathcal{A}}$ determines a point
counting measure as it counts how frequently each letter $\ga$ in the
alphabet $\mathcal{A}$ appears asymptotically. One then applies the
following lemma~\cite{JP86} to verify aperiodicity. Another property
that is used is the so-called height of the substitution $S$, which
can be calculated using \cite[Def.~6.1]{MQ10}.

\begin{lem}[Pansiot's Lemma]
  A primitive\/ $q$-substitution\/ $S$ which is one-to-one on\/
  $\mathcal{A}$ is aperiodic if and only if\/ $S$ has a letter with at
  least two distinct neighbourhoods.
\end{lem}

Bartlett's algorithm employs the bi-substitution of the substitution $S$,
which is defined as follows.

\begin{dfn}
  Let $S$ be a $q$-substitution on the alphabet $\mathcal{A}$. The
  substitution product $S\otimes S$ is a $q$-substitution on
  $\mathcal{A}\mathcal{A}$ (the alphabet formed by all pairs of
  letters in $\mathcal{A}$) with configuration $R\otimes R$ whose
  $j$-th instruction is the map
\[
    (R\otimes R)_{j}\!:\, \mathcal{A}\mathcal{A}\longrightarrow 
     \mathcal{A}\mathcal{A}\quad\text{with}\quad
    (R\otimes R)_{j}\!:\,\al\ga\longmapsto R_j(\al)R_j(\ga).
\]
The substitution $S\otimes S$ is called the \emph{bi-substitution} of $S$.
\end{dfn}

The Fourier coefficients $\widehat{\Si}$ of the correlation measures
$\Si$ can then be obtained using following theorem of
Bartlett~\cite{AB14}.

\begin{thm}
\label{thm:AB1}
Let\/ $S$ be an aperiodic\/ $q$-substitution on\/ $\mathcal{A}$. Then, for
$p\in\N$, we have
\[
   \widehat{\Si}(k)\, =\, \frac{1}{q^{p}}\sum_{j\in[0,q^{p})}R_{j}^{p}\otimes 
   R_{j+k}^{p}\,\widehat{\Si}\lfloor j+k\rfloor_{p} \, =\, 
   \lim_{n\to\infty}\frac{1}{q^{n}}\sum_{j\in[0,q^{n})}R_{j}^{n}\otimes 
   R_{j+k}^{n}\,\widehat{\Si}(0),
\] 
where\/ $\lfloor j+k\rfloor_{p}$ is the quotient of\/ $j+k$ under
division modulo\/ $q^{p}$. Here $R_j\otimes R_{j+k}$ is the Kronecker
product of the instruction matrices at position $j$ and $j+k$.
\end{thm}

Together with the above theorem and Michel's
lemma~\cite[Thm.~2.1]{AB14}, we have
\[
    \widehat{\Si}(0)\, =\, \sum_{\ga\in\mathcal{A}}u\cdot e_{\ga\ga},
\] 
where in general $e_{\alpha\beta}$ is the standard unit vector in
$\C^{\mathcal{A}^{2}}$ corresponding to the word $\alpha\beta$.
Define the $p$-th carry set to be $\Delta_p(k):=\{j\in
[0,q^{p}):j+k\neq [0,q^{p})\}$. As a consequence of the above theorem,
we have the following expression,
\begin{equation}
\label{equation:1}
  \widehat{\Si}(1)\, =\, 
  \left(qI-\sum_{j\in\Delta_1(1)}R_j\otimes R_{j+1}\right)^{-1}
  \sum_{j\notin\Delta_1(1)} R_j\otimes R_{j+1}\, \widehat{\Si}(0).
\end{equation}

We then use the following proposition~\cite[Prop.~2.2]{AB14} to
compute the bi-substitution and to partition the alphabet into its
ergodic classes and a transient part.

\begin{prop}
\label{prop:AB2}
Let\/ $S$ be a substitution of constant length on\/
$\mathcal{A}$. Then there is an integer\/ $h>0$ and a partition of the
alphabet\/ $\mathcal{A}=E_1\sqcup\cdots\sqcup E_k\sqcup T$ so that
\begin{enumerate}
\item[(i)] $S^{h}\!:\, E_{j}\to E_{j}^{+}$ is primitive for each\/ $1\leq j\leq K$,
\item[(ii)] $\ga\in T$ implies\/ $S^{h}(\ga)\notin T^{+}$,
\end{enumerate}
where\/ $\sqcup$ denotes the disjoint union, $E_j$ its ergodic classes
and\/ $T$ the transient part. $E_{j}^{+}$ and $T^{+}$ are the words
formed by elements of the ergodic classes and transient part,
respectively.
\end{prop}

We define the \textit{spectral hull} $K(S)$ of a $q$-substitution to
be 
\[
  K(S)\, :=\, \{v\in\C^{\mathcal{A}^{2}}: C_{S}^{t}v=qv \text{ and }
v\geq 0\},
\] 
and denote the \textit{extreme rays} of $K(S)$ by $K^{*}$. Here,
$C_S=\sum_{j}R_j\otimes R_j$, the sum of the Kronecker product of the
instruction matrices at each position $j$. Using the following lemma
of Bartlett~\cite{AB14} and enforcing strong semi-positivity, we 
obtain the extreme rays $K^{*}$ of the spectral hull $K(S)$. Here,
we use the notation
$\vec{E}:=\sum_{\ga\delta\in E}e_{\ga\delta}\in\C^{\mathcal{A}^{2}}$.

\begin{lem}
\label{lem:AB3}
A vector\/ $v\in\mathbb{C}^{\mathcal{A}^{2}}$ satisfies\/ $v\in K(S)$
if and only if
\[
  v\, =\, V+P_{T}(QI-P_{T}C_{S}^{t})^{-1}P_{T}C_S^{t}V 
  \quad\text{and}\quad v\geq 0,
\]
where\/ $V=\sum_{j}w_j\vec{E}_{j}$ with\/ $w_j\in\C$, and where\/ $P_{T}$
is the standard projection onto the transient pairs\/ $T$ of\/
$\mathcal{A}^{2}$.
\end{lem} 

Finally, the \emph{maximal spectral type} is given by
\begin{equation}\label{eq:mst}
   \sigma_{\textnormal{max}}\, \sim\, \omega_q*\sum_{w\in
  K^{*}}\la_{w},
\end{equation}
where $\omega_q$ is a probability measure supported by the $q$-adic
roots of unity. For each $w\in K^{*}$, we compute
\[
  \widehat{{\la}_{w}}(k)\, =\, w\widehat{\Si}(k).
\]
If $\widehat{{\la}_{w}}(k)$ is periodic in $k$, then $\la_{w}$ is a
pure point measure, if $\widehat{\la_{w}}(k)=0$ for all $k\neq 0$,
then $\la_w$ is Lebesgue measure. Otherwise, $\la_{w}$ is purely
singular continuous. Thus, the maximal spectral type is completely
characterized by this algorithm.

\section{The Rudin--Shapiro-like sequence}

The Rudin--Shapiro-like sequence of \cite{PNR15} can be described by
the following substitution rule
\begin{equation}\label{eq:subst}
  S^{}_{\text{RSL}^{}}\!:\; 0\mapsto 01,\quad 1\mapsto 20,\quad 2\mapsto 13,
  \quad 3\mapsto 32,
\end{equation}
on four letters. This is similar to the Rudin--Shapiro case, where the
binary sequence is also obtained from a four-letter substitution rule,
after applying a reduction map. We apply the recoding $0,1\to +1$ and
$2,3\to -1$. Both letters $\pm 1$ then are equally frequent, so 
we are in the balanced weight case.

In the remaining of this article, we are going to apply Bartlett's
algorithm to prove the following result.

\begin{thm}
  The (balanced weight) sequence $S^{}_{\text{RSL}^{}}$ has
  purely singular continuous diffraction spectrum.
\end{thm}
\begin{proof}
  The instruction matrices and the substitution matrix can be read off
  from the substitution rule of Equation~\eqref{eq:subst} and are given by
\[
  R_0= 
\begin{pmatrix}
1 & 0 & 0 & 0 \\
0 & 0 & 1 & 0\\
0 & 1 & 0 & 0\\
0 & 0 & 0 & 1
\end{pmatrix},\quad 
R_1=
\begin{pmatrix}
0 & 1 & 0 & 0 \\
1 & 0 & 0 & 0\\
0 & 0 & 0 & 1\\
0 & 0 & 1 & 0
\end{pmatrix}
\quad\textnormal{and}\quad
M_{\textnormal{RSL}}^{}=
\begin{pmatrix}
1 & 1 & 0 & 0 \\
1 & 0 & 1 & 0\\
0 & 1 & 0 & 1\\
0 & 0 & 1 & 1
\end{pmatrix}.
\]
As $M_{\textnormal{RSL}}^{3}\gg 0$, the substitution is primitive.
The third iterate of the seed $0$ is $01201301$, which shows that the
letter $0$ can be preceded by $2$ or by $3$, and
that the letter $1$ can be followed by either $2$ or by $3$. Hence
both $0$ and $1$ have two distinct neighbourhoods and, by Pansiot's
Lemma, the sequence is aperiodic.

In accordance with the Perron--Frobenius theorem, we find
$\la_{\textnormal{PF}}=2$ and $u=\frac{1}{4}(1,1,1,1)$ for the
eigenvalue and statistically normalized eigenvector of
$M_{\textnormal{RSL}}^{}$. By applying Theorem~\ref{thm:AB1}, we obtain
$\widehat{\Si}(0)=\frac{1}{4}\sum_{\al\in\mathcal{A}}e_{\al\al}$. As
we are dealing with a length two substitution, we have
$\Delta_{1}(1)=\{1\}$. Using Equation~\eqref{equation:1}, we find that 
\[
  \widehat{\Si}(1)\, =\, 
  \left(0,\frac{1}{6},0,\frac{1}{12},0,0,\frac{1}{12},
  \frac{1}{6},\frac{1}{6},\frac{1}{12},0,0,\frac{1}{12},0,
   \frac{1}{6},0\right).
\]
We then proceed to compute $\widehat{\Si}(k)$ for any $k\geq 2$.

By using Proposition~\ref{prop:AB2}, we calculate the ergodic
decomposition of the bi-substitution $S_{\textnormal{RSL}^{}}\otimes
S_{\textnormal{RSL}^{}}$ to obtain
\[  E_1=\{00,11,22,33\}, \quad 
    E_2=\{03,12,21,30\}, \quad 
    E_3=\{01,02,10,13,20,23,31,32\}
\] 
as the ergodic classes. In our case, the transient part turns out to
be empty. Note that $E_1$ and $E_2$ contain exactly the same elements
as the two corresponding ergodic classes of the Rudin--Shapiro
sequence.

Using Lemma~\ref{lem:AB3}, and taking into account that we have an
empty transient part $P_T^{}=0$, it follows that
\[
   v\, =\,\left(\begin{matrix}
   w_1 & w_3 & w_3 & w_2 \\
   w_3 & w_1 & w_2 & w_3\\
   w_3 & w_2 & w_1 & w_3\\
   w_2 & w_3 & w_3 & w_1
   \end{matrix} \right).
\]
We then diagonalize the matrix $v$,  
\[
  v_d\, =\, \left(\begin{matrix}
   w_2+w_1+2w_3 & 0 & 0 & 0\\
   0 & w_2+w_1-2w_3 & 0 & 0\\
   0 & 0 & -w_2+w_1 & 0\\
   0 & 0 & 0 & -w_2+w_1
   \end{matrix}\right).
\]
Setting $w_1=1$, strong semi-positivity is equivalent
to $w_{2}$ and $w_{3}$ satisfying the following three inequalities,
\[
  1-w_2\geq 0,\quad
   1+w_2+2w_3\geq 0,\quad
   1+w_2-2w_3\geq 0.
\]
The extreme points are given by the solutions $(w_1,w_2,w_3)=(1,1,1)$,
$(w_1,w_2,w_3)=(1,1,-1)$ or $(w_1,w_2,w_3)=(1,-1,0)$. Thus, the
extremal rays are
\begin{align*}
  v_1=&(1,1,1,1,1,1,1,1,1,1,1,1,1,1,1,1), \\ 
  v_2=&(1,-1,-1,1,-1,1,1,-1,-1,1,1,-1,1,-1,-1,1)\mbox { and }\\
  v_3=&(1,0,0,-1,0,1,-1,0,0,-1,1,0,-1,0,0,1).
\end{align*}
As usual, $\la_{v_1}=\delta_{0}$ which gives rise to the pure point
component, via Equation~\eqref{eq:mst}. Using the previously computed
values of $\widehat{\Si}(k)$, one checks that $\widehat{\la_{v_2}}(k)$
and $\widehat{\la_{v_3}}(k)$ do not vanish at all positions $k\neq 0$,
which proves that there are no absolutely continuous components. One
can then easily verify that the substitution system is of trivial
height, therefore the pure point component is entirely supported by
the Dirac measure $\delta_{0}$. The other two measures are neither
absolutely continuous nor show the necessary periodicity to contribute
to the pure point part. By Dekking's theorem~\cite[Thm.~5.6]{AB14},
we thus conclude that the other two measures have to be singular
continuous. Thus, we have a purely singular continuous diffraction
spectrum in the balanced weight case (in which the pure point
component is extinguished).
\end{proof}

If we assumed that the Rudin--Shapiro-like sequence satisfied the
inequality \eqref{eq:ineq}, it would imply that the diffraction
spectrum was absolutely continuous, as a consequence of the following
result \cite[Prop.~4.9]{MQ10}.

\begin{prop}
If $\sigma$ is the unique correlation measure of the sequence
$\gamma$, $\sigma$ is the weak-$*$ limit point of the sequence of
absolute continuous measures $R_N\cdot m$, where $m$ is the Haar
measure and $R_N=\frac{1}{N}\left|\sum_{n<N}\gamma(n)e^{2\pi
  in\theta}\right|^{2}$,
\end{prop} 

Let us denote $\zeta_N=R_N\cdot m$ and suppose weak convergence to a
limit $\zeta$. Assuming that Equation~\eqref{eq:ineq} holds, it
follows that $\zeta(g)\leq C\int g\ dm$, which implies absolute
continuity. Hence, it follows from the singular diffraction that the
inequality \eqref{eq:ineq} does \emph{not} hold for the
Rudin--Shapiro-like sequence.

\section{Comparison with the Rudin--Shapiro sequence}

Let us close with a brief comparison with the Rudin--Shapiro sequence.
The following result about the Rudin--Shapiro sequence is well known;
see \cite[Ch.~10.2]{BG13} and references therein for background and details.
\begin{prop}
  The Rudin--Shapiro sequence (with balanced weights) has purely
  absolute continuous diffraction spectrum.
\end{prop}
We refer the readers to~\cite[Ex.~5.8]{AB14} to see how Bartlett's
algorithm can be employed to show the above result.

Both the RS sequence and the RSL sequence are based on (four-letter)
substitutions of constant length $q=2$ (and a subsequent reduction to
a balanced two-letter sequence), and superficially looks quite
similar, including sharing the behaviour of partial sums that we
mentioned earlier. The ergodic classes $E_1$ and $E_2$ of both
substitutions contain exactly the same elements. The elements that
form the transient part of the Rudin--Shapiro sequence are exactly the
same elements that form the third ergodic class of the
Rudin--Shapiro-like sequence. However, the values obtained from the Fourier
transform of the correlation measures differ between these two
systems.  Hence, we have two structurally different systems that
exhibit a similar arithmetic structure.

Bartlett's algorithm indicates that it may be quite difficult to
construct substitution-based sequences with absolutely continuous
diffraction spectrum, because it requires $\widehat{\la_{v}}(k)$ to
vanish for \emph{all} $k\ne 0$ for one of the extremal rays.
Intuitively, this is the case because any non-trivial correlation will
give rise to long-range correlations due to the built-in
self-similarity of the substitution-based sequence. Generically, this
property will not be fulfilled, so one should expect singular
continuous spectra to dominate, which is indeed what is observed. A
notable exception is provided by substitution sequences based on
Hadamard matrices \cite{Fra03}.

\ack The authors would like to thank Michael Baake and Ian Short for
many helpful discussions and comments on improving this paper, to Alan
Bartlett on explaining his paper and Jean-Paul Allouche for sharing
his preprint~\cite{JA16}. The first author is supported by the Open
University PhD studentship.

\end{document}